\documentclass[12pt]{amsart}

\usepackage{times}      
\usepackage{amssymb}    
\usepackage{amsmath}    
\usepackage{amsthm}
\usepackage{tikz-cd}

\newtheorem{theorem}{Theorem}[section]
\newtheorem*{mainthm}{Main Theorem}
\newtheorem{proposition}[theorem]{Proposition} 
\newtheorem{corollary}[theorem]{Corollary} 
\newtheorem{lemma}[theorem]{Lemma} 

\theoremstyle{definition}
\newtheorem{definition}[theorem]{Definition} 
\newtheorem*{acknowledge}{Acknowledgements}
\newtheorem{example}[theorem]{Example}

\def\ccc{\mathbb{C}}
\def\zzz{\mathbb{Z}}
\def\rrr{\mathbb{R}}
\def\nnn{\mathbb{N}}

\def\co{\colon\thinspace}
\def\ooo{\mathcal{O}}
\def\od{(\mathcal{O},d)}
\def\uod{E_{\ooo,d}}
\def\uodch{\check{E}_{\ooo,d}}
\def\bod{B_{\ooo,d}}
\def\bodch{\check{B}_{\ooo,d}}
\def\cozb{C^{\ooo}(Z,B_{\ooo,d})}

\def\dist{\mathrm{dist}}

\begin{document}

\title[$G$--actions with close orbit spaces]{$G$--actions with close orbit spaces}

\author[Harvey]{John Harvey}
\address{University of Notre Dame, Department of Mathematics, Notre Dame, Ind. 46556, U.S.A.}
\email{jharvey2@nd.edu}

\date{\today}

\begin{abstract}
The classification of $G$--spaces by Palais is refined for the case where the orbit space satisfies certain mild topological hypotheses.

It is shown that when a sequence of such orbit spaces is ``close" to a limit orbit space, in some suitable sense, within a larger ambient orbit space, the $G$--spaces in the tail of the sequence are strongly equivalent to the limit $G$--space.
\end{abstract}
\maketitle

\section{Introduction}

In 1960 Palais described the classification of $G$--spaces for compact Lie groups $G$. Up to equivalence, $G$--spaces with finitely many orbit-types over a given orbit space are classified by the path-components of a particular function space. 

This result is a generalization of the classification of principal $G$--bundles over a base $Z$, which is by the path-components of $C(Z,BG)$. However, the function space in Palais' classification is more complicated. In a further generalization by Ageev \cite{Ageev}, it was shown that the finiteness condition on orbit-types can be removed.

In Section \ref{s:refinement} of this paper it is shown that, still assuming the $G$--space to have only finitely many orbit-types, and assuming further that the orbit space is compact metrizable and that the embedding of each subset of the orbit space corresponding to a given orbit type is ``reverse tame''  this function space is, in fact, locally path-connected. In consequence, the $G$--spaces are classified by the connected components of the function space.

The refinement is then applied in Corollary \ref{c:closeorbitspaces} to obtain a Covering Sequence Theorem. This shows that when a sequence of invariant subspaces of an ambient $G$--space $X$ have orbit spaces which are ``close" in some sense to a fixed orbit space inside $X/G$ and satisfy this tameness requirement, then the $G$--spaces in the tail of the sequence are pairwise equivalent. 

In later work \cite{hequi}, this result is shown to be important in understanding equivariant convergence of metric spaces in the Gromov--Hausdorff sense.

\begin{acknowledge}This research was carried out as part of the author's dissertation project, under the ever-helpful advice of Karsten Grove. The author was supported in part by a grant from the U.S. National Science Foundation.\end{acknowledge}

\section{The Classification of $G$--spaces}

This section recalls Palais' original work \cite{palais}.

Assume that $G$ is a compact Lie group. Let $G$ act on $X$ by homeomorphisms. If $H$ is the isotropy of a point $p \in X$, then all conjugates of $H$ appear as isotropy groups in the orbit $G(p)$. Denote the set of all conjugates of $H$ by $(H)$. The orbit $G(p)$ is then said to be an orbit of type $(H)$. 

The union of all orbits of type $(H)$ in $X$ will be denoted by $X_{(H)}$. Its image in the orbit space $X/G$ will be denoted by $(X/G)_{(H)}$.

\subsection{$\ooo$--spaces}

The analogy to the base of a principal bundle is the orbit space of a group action. However, the base of a bundle is just a topological space, while an orbit space contains other data. An $\ooo$--space should be thought of as an orbit space which has forgotten its group action, but still retains these data.

\begin{definition}Let $G$ be a compact Lie group, and let $\mathcal{O}$ be a collection of $G$--orbit-types. Then an \emph{$\ooo$--space} is defined to be a locally compact, second countable space $X$ together with a partition $\{ X_{(H)}\}_{(H) \in \ooo}$ of $X$ indexed by $\ooo$ such that, for each $(H) \in \ooo$, $\cup \{X_{(K)} \vline (K) \leq (H)\}$ is open.\end{definition}

\begin{definition}An \emph{$\ooo$--map} between two $\ooo$--spaces $X$ and $Y$ is a continuous map $f\co  X \to Y$ such that $f(X_{(H)}) \subset Y_{(H)}$ for all $(H) \in \ooo$.\end{definition}

These definitions give a category having all $\ooo$--spaces for objects and $\ooo$--maps for morphisms. The isomorphisms in this category
will be called \emph{$\ooo$--equivalences}. If $f_0\co  X \to Y$ is an $\ooo$--map, an $\ooo$--homotopy of $f_0$ is an $\ooo$--map $f\co X \times I \to Y$ such that $f(x,0)=f_0(x)$. Writing $f_t(x)$ for $f(x,t)$, $f_0$ and $f_1$ are said to be \emph{strongly $\ooo$--homotopic}. Say that two $\ooo$--maps $f_0$ and $f_1$ from $X$ to $Y$ are \emph{weakly $\ooo$--homotopic} if there is an $\ooo$--equivalence $h$ of $X$ with itself such that $f_0$ and $f_1 \circ h$ are $\ooo$--homotopic.

Let $C^{\mathcal{O}}(X,Y)$ be the set of $\ooo$--maps from $X$ to $Y$ endowed with the compact-open topology. Let $[X,Y]^{\mathcal{O}}$ be the set of strong $\ooo$-homotopy classes of these maps -- these are the path-components of $C^{\mathcal{O}}(X,Y)$. The group of $\ooo$--equivalences of $X$ with itself acts on $[X,Y]^{\mathcal{O}}$ and the orbit space is the set of weak $\ooo$--homotopy classes of maps.

Some control on the dimension of the orbit space is also needed for the classification. Let $\ooo$ be a finite set of $G$--orbit-types. An $\ooo$--dimension function is a function $d\co  \ooo \to \{ -1, 0, 1, \ldots \}$. If an $\ooo$--space $X$ satisfies $\dim(X_{(H)}) \leq d((H))$ for each $(H) \in \ooo$, then $X$ will be called an $\od$--space.

\subsection{$G$--spaces}

A $G$--space will be a locally compact, second countable space acted on by $G$ by homeomorphisms. 

\begin{example}
Let $X$ be a $G$--space, and let $\ooo$ be the set of orbit-types of the $G$--action. Then $X/G$ is an $\ooo$ space.
\end{example}

Given an $\ooo$--space $Z$, a $G$--space over $Z$ is a triple $(X, Z, h)$ where $X$ is a $G$--space and $h\co  X/G \to Z$ is an $\ooo$--equivalence. If $(X', Z, h')$ is another $G$--space over $Z$, and $f\co  X \to X'$ is a $G$--equivariant homeomorphism, $f$ is called a \emph{weak equivalence} of $G$--spaces over $Z$. The map $f$ induces a map $\bar{f}\co X/G \to X'/G$, and in the case that $h' \circ \bar{f} \circ h^{-1}$ is the identity map on $Z$ $f$ is called a \emph{strong equivalence}.

Given a $G$--space over an $\ooo$--space $Y$ and a second $\ooo$--space $X$, any map in $C^{\mathcal{O}}(X,Y)$ induces a unique $G$--space over $X$ by pull-back. Strongly homotopic maps induce strongly equivalent $G$--spaces, while weakly homotopic maps induce weakly equivalent $G$--spaces.

\subsection{Palais' construction}

In analogy to universal bundles, Palais constructs universal $G$--spaces. These spaces are reduced joins. 

The \emph{join} of several $G$--spaces $X_1, X_2, \ldots , X_l$ is denoted by $X_1 \circ X_2 \circ \ldots \circ X_l$. It is obtained from the product $X_1 \times [0,1] \times X_2 \times [0,1] \times \ldots \times X_l \times [0,1]$ by identifying $$((x_1,t_1),(x_2,t_2),\ldots , (x_l,t_l)) \sim ((y_1,s_1),(y_2,s_2),\ldots , (y_l,s_l))$$ when $t_i=s_i$ for all $i = 1, \ldots l$ and $x_i = y_i$ if $t_i \neq 0$. A natural $G$--action on the join is given by letting $G$ act trivially on the interval factors.

The \emph{reduced join}, $X_1 * X_2 * \ldots * X_l$ is a subset of the full join, and is obtained by requiring that if a point
$$p  = ((p_1,t_1),(p_2,t_2),\ldots(p_l,t_l)) \in X_1 \circ X_2 \circ \ldots \circ X_l$$
has $G_p$ as its isotropy group then there must be some $i$ such that $t_i \neq 0$ and $p_i \in X_i$ has the same isotropy group $G_p$. Points not satisfying this condition are removed to give $X_1 * X_2 * \ldots * X_l$.

Consider a finite collection of $G$--orbit-types $\ooo = \{ (H_1), \ldots , (H_r) \}$ and a dimension function $d$. Let $N(H_i)$ be the normalizer of $H_i$ in $G$, and suppose $N(H_i)/H_i$ is $m_i$--connected. Then for any 
$$
k_i \geq \frac{d((H_i)) + 2}{m_i + 2}
$$
define
$$
\uodch = (G/H_1)^{\circ k_1} \circ \cdots (G/H_i)^{\circ k_i} \circ \cdots (G/H_r)^{\circ k_r}
$$
and let $\uod$ be the corresponding reduced join. 

The $G$--space $\uod$ is a universal $\od$--space. The reader may find more information on the properties of a ``universal" $G$--space in \cite{palais}, but its orbit space will be more relevant for this paper.

Let $\bodch$ be $\check{E}_{\ooo,d+1} / G$, and let $\bod$ be $E_{\ooo,d+1}/G$. Now the original classifying result of Palais can be precisely stated.

\begin{theorem}\cite{palais}
$B_{\ooo,d}$ is $(\ooo,d)$--classifying, meaning that the set of strong equivalence classes of $G$--spaces over an $(\ooo,d)$--space $Z$ is in one-to-one correspondence with the set $[Z,B_{\ooo,d}]^{\mathcal{O}}$.
\end{theorem}

The set of weak equivalence classes of $G$--spaces is obtained by allowing the self-equivalences of $Z$ to act on $[Z,B_{\ooo,d}]^{\mathcal{O}}$.

\section{Examples I}

This section provides some simple examples of how the classification theory works for actions of the circle.

\begin{example}
	Let $G=S^1$ and let $\ooo = \{ S^1, 1 \}$, in order to classify the actions of the circle without finite isotropy. Then $\bod = K \ccc P^n = (\ccc P^n \times [0,1]) / (\ccc P^n \times \{0\})$, where $n$ depends on $d(S^1)$. Let $Z$ be a conical orbit space, $KY = (Y \times [0,1]) / (Y \times \{0\})$ such that $Z_{S^1}$ is the base of the cone, $Y \times {1}$, and $Z_{\{1\}}$ is the complement.
	
	Specify a point $p \in \ccc P^n$. Any $\ooo$--map $f \co  Z \to K \ccc P^n$ may be written $f(y,r)=(g(y,r),h(y,r)) \in \ccc P^n \times [0,1]$ where $h(y,r) = 0 \iff r=1$.
	When $r=1$ the value of $g$ is not relevant. Indeed, it may not be possible to choose $g$ to be continuous at the base of the cone.
	One may easily homotope $f$ so that $h(y,r)=1-r$, and such an $\ooo$--map may be said to be in ``standard form". Now it may be homotoped further by deforming $g$ so that $g_t (y,r) = g(y,(1-t)r) $. When $t=1$, the resulting map sends the entire circle $Y \times \{r\}$ to the point $(f(o),1-r)$, where $o$ is the vertex of $Z$. This map can be homotoped once more so that $Y \times \{r\}$ is mapped to $(p,1-r)$, and it is now clear that all $\ooo$--maps between the spaces are homotopic.
	
	Therefore, up to strong equivalence, for each conical orbit space $Z$ with fixed points on the base and principal orbits in the complement, there is only one $S^1$--space. If $Z=D^2$, for example, it can only arise from the unique linear action on $S^3$ which fixes a circle and is free elsewhere.
\end{example}

\begin{example}\label{e:product}
	Again, let $G=S^1$ and $\ooo = \{ S^1, 1 \}$ so that $\bod = K \ccc P^n = (\ccc P^n \times [0,1]) / (\ccc P^n \times \{0\})$, where $n$ depends on $d(S^1)$. Let $Z$ be a product, $Z = Y \times [-1,1]$ such that $Z_{S^1} = Y \times \{-1,1\}$ and $Z_{\{1\}} = Y \times (-1,1)$.
	
	Once again, any $\ooo$--map $f \co  Z \to K \ccc P^n$ may be written $f(y,r)=(g(y,r),h(y,r)) \in \ccc P^n \times [0,1]$ where $h(y,r) = 0 \iff r=-1 \textrm{ or }1$. One may easily homotope $f$ to a ``standard form", here meaning that $h(y,r)=1-|r|$. In fact, given any $\ooo$--homotopy $F\co Z \times I \to K \ccc P^n$, it can be deformed to a homotopy $\tilde{F}$ such that each map $f_t = \tilde{F}(\cdot,t)$ is in standard form. By restricting attention to such homotopies, it is clear that $[Z,K \ccc P^n]^{\mathcal{O}} \leftrightarrow [Y,\ccc P^n]$.
	
	For example, if $Y=S^2$, the $S^1$--spaces are classified up to strong equivalence by elements of $\pi_2(\ccc P^2) \cong \zzz$. To a given integer $p$ corresponds a double mapping cylinder where both maps are given by  the orbit space projection of the free circle action on the lens space $L_{p,1} \to S^2$. Here $L_{0,1}$ represents $S^2 \times S^1$ and $L_{1,1}$ represents $S^3$. The spaces given by $p$ and $-p$ are weakly equivalent---the maps are weakly homotopic via an orientation-reversing automorphism of the orbit space $S^2 \times [-1,1]$.
\end{example}

\section{Refinement of the classification}\label{s:refinement}

The Main Theorem of this work is the refinement of Palais' classification, showing that strong equivalence classes are in bijection with the connected components of $\cozb$. This is equivalent to showing that the function space $\cozb$ is locally path-connected.

In this section the appropriate properties of both the codomain, $\bod$, and the domain, $Z$, will be investigated, before the result is proved.

\subsection{Properties of the codomain}

Such connectedness statements are usually dependent primarily on the codomain of the function space.
In this case, the codomain $\bod$ has a very nice geometric structure.

The space $\uodch$ is a join of smooth homogeneous manifolds.
If each of these manifolds is given a $G$--invariant Riemannian metric, $\uodch$ has a natural metric given by the spherical join of the manifolds.
This metric is $G$--invariant, and so it induces a metric on $\bod$.
Assume from now on that $\bod$ is endowed with such a metric.

\begin{lemma}
	The space $\bodch$ is a Whitney stratified space, and $\bod$ is a stratified subspace.
\end{lemma}

\begin{proof}
	The full join $\check{E}_{\ooo,d+1}$ is a Whitney stratified space, with the strata given by the subspaces $G/H_i$ and their joins. 
	The group $G$ preserves the strata, so that the quotient $\bodch$ is also Whitney stratified, the strata given by the images of the strata of $\check{E}_{\ooo,d+1}$ as well as by orbit type.
	
	To obtain $\bod$ from $\bodch$, orbits of type $(L)$ are removed if they are in the image of the subjoin of all $(G/H_i)$ with $(H_i) \neq (L)$.
	The removed orbits therefore are a union of connected components of strata, and so what remains, $\bod$, is a stratified subspace.
\end{proof}

Every orbit type in $\bod$ is embedded with a tubular neighborhood as follows. 
Let $(K) \in \ooo$. 
Then the preimage in $\uodch$ of $B_{(K)}$ is contained in the join of all $G/H$ such that $(K) \leq (H)$, and it is open in that subset after passing to the reduced join. 
Normal to this in $\uodch$ is the cone on the join of all $G/H$ such that $(K) \nleq (H)$. 
Let $\Sigma$ be the corresponding reduced join. 
Then there is an open neighborhood $T_{(K)}$ of $B_{(K)} \subset \bod$ which is a fiber bundle over $B_{(K)}$, with fiber given by the cone on $\Sigma / K$. 
Let $p_{(K)} \co T_{(K)} \to B_{(K)}$ be the projection of the fiber bundle. 
For any $\epsilon > 0$, it is possible to choose $T_{(K)}$ so that the diameter of the fibers of $p_{(K)}$ is at most $\epsilon$, and so that the projection $p_{(K)}$ is a short map.

Note that for each $(K) \in \ooo$, $B_{(K)} \subset \bod$ is an absolute neighborhood retract (ANR). 
Suppose that $f \co X \to Y$ is a map from a space $X$ into a metric ANR $Y$, and that $f(X) \subset K \subset Y$ for some compact $K$. 
Then if $g$ is another map sufficiently close to $f$, $g$ and $f$ must be homotopic. 
This can be seen by embedding $Y$ in a linear space, constructing a linear homotopy from $f$ to $g$, and then retracting the homotopy to $Y$. 
The diameter of the homotopy is clearly controlled.
See, for example, van Mill for a proof \cite[Theorem 4.1.1]{vanmill}.

Furthermore, if a controlled homotopy has already been defined on some closed $A \subset X$, and $X$ is normal, it can be extended in a controlled way. 
This is done by interpolating in a linear fashion between the retract of the linear homotopy and the given homotopy using an Urysohn function. The author could not locate a proof of this fact, but the equivariant version was proved by Antonyan et al \cite{antonyan}.

Some notions relating to homotopies of maps into stratified spaces are needed.

Let $Y$ be a stratified space. A map $f \co X \times A \to Y$ will be called \emph{stratum-preserving} along $A$ if, for each $x \in X$, $f(x \times A)$ lies in a single stratum of $Y$.

A \emph{stratum-preserving homotopy} from $X$ is a map defined on $X \times [0,1]$ which is stratum preserving along $[0,1]$.

Let $X$ and $Y$ be stratified spaces and let $p \co X \to Y$ be a map. Then $p$ is a \emph{stratified fibration} provided that given any space $Z$ and any commuting
diagram
\begin{center}
	\begin{tikzcd}
		Z	\arrow{r}{f} \arrow{d}{\times 0}	& X \arrow{d}{p}	\\
		Z \times [0,1]	\arrow{r}{F}	&	Y
	\end{tikzcd}
\end{center}
where $F$ is a stratum-preserving homotopy, there exists a  stratum-preserving homotopy $\tilde{F} \co Z \times [0,1] \to X$ such that $\tilde{F}(z,0) = f(z)$ for each $z \in Z$ and $p \tilde{F} = F$.

\begin{example}\label{e:stratumbundle}Let $E$ be a stratified space, and let $p \co E \to B$, with $B$ paracompact. If $p$ is a fiber bundle, and if local trivializations of $p$ can be chosen to respect the stratification of $E$, then $p$ is a stratified fibration. This is because the usual proof of the Covering Homotopy Theorem (see Huebsch, for example \cite{huebsch1,huebsch2}) operates in a stratum-preserving way. \end{example}

\subsection{Properties of the domain}

Usually the conditions on the domain are much less stringent.
However, in this case the functions are restricted to those which respect the partitions of the spaces, and so some control is necessary on how the domain is partitioned.
It is sufficient that the partition of the space $Z$ satisfy a certain ``tameness'' condition.
The precise condition, as given below, is best compared with the notion of ``reverse tameness'' used by Quinn \cite{quinn} in his study of homotopically stratified sets.
It is somewhat weaker than Quinn's definition, except for the inclusion of a stratum-preserving requirement.

It will be convenient to move from the partition on $Z$ to the related filtration.
For the purposes of this paper, a filtration may have an index set which is only partially ordered.
For $(H) \in \ooo$, write $Z_{\geq (H)}$ for the union of all $Z_{(K)}$ such that $(K) \geq (H)$. 
$Z_{\geq (H)}$ is a closed set.
These sets $Z_{\geq (H)}$ make up a filtration of $Z$ indexed by the partially ordered set $\ooo$, but with the reverse ordering.

\begin{definition}
	Let $X$ be a filtered set (with the filtration indexed by a set which is possibly only partially ordered).
	The filtration is said to be \emph{tame} if for each $Y \subset X$ which is a union of elements of the filtration
	and for each open cover $\mathcal{U}$ of $X$
	there are a neighborhood $V$ of $Y$ 
	and a homotopy $h \co (X - Y) \times I \to X - Y$ satisfying:
	\begin{enumerate}
		\item $h$ is the identity on $(X - Y) \times \left\lbrace 0 \right\rbrace$, 
		\item $h((X - Y) \times \left\lbrace 1\right\rbrace ) \subset X - V$,
		\item $h$ is a $\mathcal{U}$--homotopy, and 
		\item $h$ preserves every member of the filtration on $X - Y$. 
	\end{enumerate}
\end{definition}

When the filtration of an $\od$--space is tame, it will be called a \emph{tame $\od$--space}.

Tameness is a local property. Say that the filtration is locally tame if, for each closed $Y \subset X$ as above, and for each $y \in Y$, there is an open set $U_y$ containing $y$ so that for each open cover $\mathcal{U}$ of $X$ there is a filtration-preserving $\mathcal{U}$--homotopy $r \co (U_y - Y) \times [0,1] \to X-Y$ deforming $U_y - Y$ into $X-V$ for some open $V \supset Y$.

\begin{proposition}\label{p:localtameness}Let $X$ be a compact metrizable space with a filtration. The filtration is tame if and only if it is locally tame.\end{proposition}

\begin{proof}
	Endow $X$ with a metric, and replace the given cover $\mathcal{U}$ with a  cover by $\epsilon$--balls, and aim to construct an $\epsilon$--homotopy.
	
	Cover $X$ with finitely many open sets $U_1, \ldots, U_N$ so that $Y$ is tame in each $U_i$.
	Choose continuous functions $a_i \co X \to [0,1]$ so that the support of each $a_i$ is in $U_i$ and $\Sigma_i a_i = 1$.
	Let $r_i \co (U_i - Y) \times [0,1] \to X-Y$ be an $\frac{\epsilon}{N}$--homotopy deforming $U_i - Y$ into $X-V$ in a stratum-preserving manner for some open $V \supset Y$.
	(Since $N$ is finite, $V$ may be assumed not to depend on $i$.)
	By an appropriate choice of $r_i$, one may assume further that $r_i ((U_i - Y) \times [\frac{1}{N+1} , 1]) \subset X-V$ and that $a_j \circ r_i$ is a $\frac{1}{(N+1)^3}$--homotopy for each $j$.  
	Extend each $r_i$ over $X \times \left\lbrace 0\right\rbrace $ by the identity.
	
	New homotopies $R_i \co (X - Y) \times [0,1] \to X-Y$ can be constructed by $R_i(x,t) = r_i (x, a_i(x)t)$. Write $R^j \co (X - Y) \times [0,1] \to X$ for the homotopy given by concatenating $R_1, \ldots R_j$. It is claimed that $R^N$ is the required deformation.
	
	Certainly since each $r_i$ is stratum-preserving, each $R_i$ is also, and so is each $R^j$. Since each $r_i$ is an $\frac{\epsilon}{N}$--homotopy, $R^N$ is an $\epsilon$--homotopy. It remains only to show that $R^N(X-Y) \times \left\lbrace 1 \right\rbrace) \subset X-U$ for some open neighborhood $U$ of $Y$.
	
	For each $q \in X-Y$, there is some $i$ so that $a_i (q) \geq \frac{1}{N}$. 
	Because each of the homotopies $r_k$ changes the value of $a_i$ by no more than $\frac{1}{(N+1)^3}$, there is some $k \leq N$ so that $a_k (R^{k-1}(q,1)) > \frac{1}{N} - \frac{1}{(N+1)^2} > \frac{1}{N+1}$ and hence $R^k (q,1) \in X-V$. 
	In other words, every $q \in X-Y$ enters the compact subset $X-V$ at some point in the construction of $R^N$.
	The homotopy will continue to deform the subset $X-V$, but its image must remain compact.
	
	It follows that $R^N$ deforms $X-Y$ into a compact subset, and its complement is the desired $U$.
\end{proof}

To provide for interpolation between homotopies, it is necessary that the domain be normal, so that Urysohn functions may be constructed.
In fact, it is necessary that certain non-compact subspaces of the domain be normal as well, so the domain will be assumed to be metrizable.

\subsection{Proof of the Main Theorem}

\begin{mainthm}
	Let $\mathcal{O}$ be a finite collection of $G$--orbit-types with dimension function $d$ and let $Z$ be a compact metrizable tame $(\ooo,d)$--space. Then the set of strong equivalence classes of $G$--spaces over $Z$ is in one-to-one correspondence with the set of connected components of $\cozb$.
\end{mainthm}

\begin{proof}

	Let $f \in \cozb$. The aim is to show that there is some $\delta > 0$ so that if $g \in \cozb$ is $\delta$--close to $f$ it is $\ooo$--homotopic to $f$.
	
	The proof proceeds in two steps. The first step is to construct local homotopies near each part of $Z$. The second step is to patch these homotopies together.
	
	\clearpage
	\textbf{Step I}
	
	This part of the proof will show that given any $\epsilon > 0$ there are an open neighborhood $N_{(K)}$ of $Z_{(K)}$ and some $\delta_{(K)} > 0$ such that if $g \in \cozb$ is $\delta_{(K)}$--close to $f$ then $\left. p_{(K)} \circ f \right|_{N_{(K)}}$ is $\epsilon$--homotopic to $\left. p_{(K)} \circ g \right|_{N_{(K)}}$. 
	
	Let $A$ be $\cup \left\lbrace Z_{(H)} : (H) > (K) \right\rbrace $. 	By tameness, there are a neighborhood $V$ of $A$ and some stratum-preserving deformation $h \co (Z-A) \times [0,1] \to Z-A$ so that $f \circ h$ is an $\frac{\epsilon}{4}$--homotopy and $h_1 (Z-A) \subset Z-V$.
	
	The image $p_{(K)} \circ f( (Z-V) \cap Z_{(K)})$ is contained in a compact subset of $B_{(K)}$.
	It follows that there is an open neighborhood of $Z_{(K)}$ in $Z-V$, the image of which is also contained in a compact subset of $B_{(K)}$.
	This open neighborhood contains $h_1 (Z_{(K)})$, and so it also contains $h_1 (N)$, where $N$ is some open neighborhood of $Z_{(K)}$.
	
	Since $T_{(K)}$ is an open neighborhood of $B_{(K)}$, its preimage under any continuous map is an open neighborhood of $Z_{(K)}$. 
	Its preimage under $f \circ h$ is an open neighborhood of $Z_{(K)} \times [0,1]$. 
	By compactness of $[0,1]$, there is then an open neighborhood $N_f$ of $Z_{(K)}$ so that $f \circ h_t (N_f) \subset T_{(K)}$ for all $t$.
	Similarly there is a neighborhood $N_g$.
	Choose $N_{(K)} = N \cap N_f \cap N_g$.
	
	By using that $B_{(K)}$ is an ANR, that $p_{(K)} \circ f \circ h_1 (N_{(K)})$ is contained in a compact subset of $B_{(K)}$, and that $p_{(K)}$ is a short map, there is some $\delta_{(K)}$ so that if $f$ and $g$ are $\delta_{(K)}$--close there is an $\frac{\epsilon}{4}$--homotopy between $p_{(K)} \circ f \circ h_1$ and $p_{(K)} \circ g \circ h_1$.
	The homotopy $p_{(K)} \circ f \circ h$ is also an $\frac{\epsilon}{4}$ homotopy. Choosing $\delta_{(K)} < \frac{\epsilon}{4}$ in addition, $p_{(K)} \circ g \circ h$ will be an $\frac{\epsilon}{2}$ homotopy. Concatenating the three homotopies, one obtains the desired $\epsilon$--homotopy.
	
	\textbf{Step II}
	
	This step of the proof is inductive. For any minimal $(K) \in \ooo$, Step I shows that for any $\epsilon>0$, $\delta$ may be chosen small enough to given an $\epsilon$--homotopy between $f$ and $g$ from $Z_{(K)} \times [0,1] \to B_{(K)}$.
	
	The inductive step is to show the following. Let $C = \left\lbrace Z_{(H)} : (H) \leq (K) \right\rbrace $. Suppose that for any $\epsilon' > 0$, $\delta$ may be chosen small enough so that if $g$ is $\delta$--close to $f$ a stratum-preserving $\epsilon'$--homotopy $L \co (C - Z_{(K)}) \times [0,1] \to \bod$ exists. Then for any $\epsilon > 0$, $\delta$ may be chosen small enough so that a stratum-preserving $\epsilon$--homotopy exists on $C$.
		
	By Step I, for any $\epsilon'' > 0$ there are $\delta > 0$ and an open neighborhood $W_{(K)}$ of $Z_{(K)}$ in $C$ (chosen with respect to some tubular neighborhood $T_{(K)}$ of $B_{(K)}$) so that if $g$ is $\delta$--close to $f$ there is an $\epsilon''$--homotopy $M \co W_{(K)} \times [0,1] \to B_{(K)}$ from $p_{(K)} \circ f$ to $p_{(K)} \circ g$.
	
	On $W_{(K)} - Z_{(K)}$, the homotopy $L$ gives rise to an $\epsilon'$--homotopy $p_{(K)} \circ L \co (W_{(K)} - Z_{(K)}) \times [0,1] \to B_{(K)}$ between the maps.
	
	Using similar methods to those described in Step I, deforming away from the set $A = \cup \left\lbrace Z_{(H)} : (H) > (K) \right\rbrace$ and picking $\epsilon'$ and $\epsilon''$ small enough (in any event, $< \frac{\epsilon}{9}$) and shrinking the neighborhood $W_{(K)}$, the two homotopies $p_{(K)} \circ L$ and $M$ can be viewed as close maps $(W_{(K)} - Z_{(K)}) \times [0,1] \to B_{(K)}$, which may be homotoped to one another by an $\frac{\epsilon}{9}$--homotopy $J \co (W_{(K)} - Z_{(K)}) \times [0,1] \times [0,1] \to B_{(K)}$. 
	Since $W_{(K)} - Z_{(K)}$ is normal, this homotopy may be assumed to be constant on $(W_{(K)} - Z_{(K)}) \times \left\lbrace 0,1 \right\rbrace  \times [0,1] \to B_{(K)}$, so that $J$ is a deformation through $\epsilon/3$--homotopies from $p_{(K)} \circ f$ to $p_{(K)} \circ g$.
	
	It is possible to lift this through $p_{(K)}$ to a stratum-preserving deformation $\tilde{J}$ of stratum-preserving homotopies from $f$ to $g$. 
	This may be achieved with the usual stratified homotopy lifting property, making use of an automorphism of $[0,1] \times [0,1]$ which brings three sides of the square to one side. 
	The lift is a deformation of $L$ on $W_{(K)}$ so that, while $\tilde{J}_0=L$ does not extend over $Z_{(K)}$, it is easy to ensure that $\tilde{J}_1$ does extend.
	
	To guarantee that $\tilde{J}_1$ extends over $Z_{(K)}$, it is enough to choose $\tilde{J}$ so that the homotopy
	$$\dist (B_{(K)} , \tilde{J}(z,t,1)) \co (W_{(K)} - Z_{(K)}) \times [0,1] \to \rrr$$
	is linear.
	However, the distance from $B_{(K)}$ can be altered
	in any desired way without affecting the fact that $\tilde{J}$ is a deformation of homotopies from $f$ to $g$ with $p_{(K)} \circ \tilde{J} = J$, so long as the image of $\tilde{J}$ does not leave $T_{(K)}$.
	
	If $T_{(K)}$ is such that the diameter of fibers of $p_{(K)}$ is bounded above by $\epsilon/3$, then the homotopies $L_t$ are all $\epsilon$--homotopies.
	
	Since $C$ is normal, it is possible to interpolate between $L_0$ and $L_1$ using an Urysohn function.
	Let $W'_{(K)} \subset W_{(K)}$ be such that the closure of $W'_{(K)}$ inside $A$ is contained within $W_{(K)}$. 
	Let $\phi \co B \to [0,1]$ vanish outside $W_{(K)}$ and be $1$ on $W'_{(K)}$.
	The required $\epsilon$--homotopy is then $\tilde{J}(z,t,\phi(z))$.
\end{proof}

\section{Examples II}

As an application, this corollary shows that when a sequence of invariant subspaces of an ambient $G$--space $X$ have orbit spaces which are ``close" in some sense to a fixed orbit space inside $X/G$, then the $G$--spaces in the tail of the sequence are pairwise equivalent.

\begin{corollary}[Covering Sequence Theorem]\label{c:closeorbitspaces}Let $X$ be a $G$--space having finitely many orbit-types, and let $Y = X/G$ be its orbit space. Let $\ooo$ be the set of orbit-types arising in $Y$, and let $d$ be a dimension function giving the dimension of each orbit-type in $Y$. Let $Z$ be a compact metrizable tame $(\ooo,d)$--space. Let $f_n \co  Z \to Y$ be a sequence of $\ooo$--maps, each of which is an $\ooo$--equivalence onto its image in $Y$. Suppose that $f = \lim_{n \to \infty} f_n$ exists, and is also an $\ooo$--equivalence onto its image. Then, for large enough $n$, the $G$--spaces induced over $Z$ by $f_n$ are strongly equivalent to that induced over $Z$ by $f$.
\end{corollary}

\begin{proof}
	Let $F\co Y \to \bod$ classify the $G$--space $X$. Then $F$ induces $\ooo$--maps $\phi_n \co  Z \to \bod$ via $f_n$ and $\phi \co  Z \to \bod$ via $f$, and these maps classify the induced $G$--spaces over $Z$. Furthermore, since $f_n \to f$ the sequence $\phi_n \to \phi$, and so by the Main Theorem the maps $\phi_n$ are $\ooo$--homotopic to $\phi$ for large enough values of $n$.
\end{proof}

\begin{example}
	Let $Z$ be a compact metrizable tame $\od$--space with finitely many orbit-types, and let $A = \{ \frac{1}{n} :  n \in \nnn \} \cup \{ 0 \}$ topologized as a subspace of $\rrr$. Then if $X$ is a $G$--space having orbit space homeomorphic to $Z \times A$, with $\pi\co  X \to Z \times A$ the projection map (identifying the orbit space with $Z \times A$), the $G$--spaces $\pi^{-1} (Z \times \{\frac{1}{n}\})$ are strongly equivalent to $\pi^{-1} (Z \times \{ 0 \})$ for large enough $n$.
\end{example}

Finally, the following examples show the necessity of an additional tameness assumption on $Z$. In Example \ref{ctrex} the orbit space $Z$ is topologically straightforward, but badly partitioned. In Example \ref{wedge} the partition is very straightforward, but the infinite topology of the orbit space violates reverse tameness.

\begin{example}\label{ctrex}
	Let $G=S^1$ and let $\ooo = \{ S^1, 1 \}$.  Let $Z$ be $[0,1] \times S^2$, and let $A = \{ \frac{1}{n} : n \in \nnn \} \cup \{ 0 \}$. Let $B_n = [\frac{1}{n+1}, \frac{1}{n} ] \times S^2$ for each $n \geq 1$. Partition $Z$ so that $Z_{S^1} = A \times S^2$ while the complement represents points of trivial isotropy. (See Figure \ref{fig:orbit}.)
	
	\begin{figure}[h]
		\centering
		\begin{tikzpicture}[scale=5]
		\draw (0.02,0) -- (0.02,1) -- (1,1) -- (1,0) -- (0.02,0);
		\foreach \n in {2,3,...,50} {
			\draw ({(1/\n)},0) -- ({(1/\n)},1);
		}
		\node at (0.4,0.5) {$B_2$};
		\node at (0.8,0.5) {$B_1$};
		\end{tikzpicture}
		\caption{The orbit space $Z$ from Example \ref{ctrex}.}\label{fig:orbit}
	\end{figure}
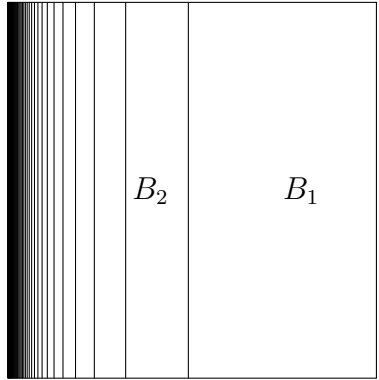

	A classifying space for $G$--spaces over $Z$ is the closed cone $K \ccc P^2$, arising as the quotient of $\uod \cong D^6$ by an action fixing an interior point and restricting to the Hopf action on the boundary. The $G$--space is classified by the maps  $B_n \to K \ccc P^2$ which take $\{\frac{1}{n+1}, \frac{1}{n} \} \times S^2$, and no other point of $B_n$, onto the fixed point . As noted in Example \ref{e:product}, a choice of such a map corresponds to choosing an element from $\pi_2 (\ccc P^2)$.
	
	Fix a metric on $K \ccc P^2$. Let $\phi_0 \co  Z \to K \ccc P^2$ be an $\ooo$--map such that, for each $n$, the restriction $\phi_0 | _{B_n}$ is not trivial in $\pi_2(\ccc P^2)$ and has image contained within a ball of radius $\frac{1}{n}$ centered on the fixed point. Inductively define a sequence of maps $\phi_n$ as follows. Assume that the map $\phi_{n-1}$ is defined, and construct $\phi_n$ from $\phi_{n-1}$ by redefining it on $B_n$ so that it becomes trivial in $\pi_2(S^2)$, but still has image contained within a ball of radius $\frac{1}{n}$ centered on the fixed point. The sequence $\phi_n$ has a limit $\phi\co  Z \to K \ccc P^2$, but $\phi$ is homotopically trivial on every $B_n$ and so is not homotopic to any of the maps in the sequence.
\end{example}

\begin{example}\label{wedge}
	Once more, let $G=S^1$ and let $\ooo = \{ S^1, 1 \}$.  
	
	Let $Z$ be $[0,1] \cup_{i=1}^{\infty} S^2$, so that the $i^{\textrm{th}}$ sphere is wedged to the interval at $\frac{1}{i}$. Let $Z_{S^1}$ be the point $\left\lbrace 0 \right\rbrace $ and let the complement represent points of trivial isotropy.
	
	Now the map from each sphere may be chosen to represent any desired homotopy class in $\pi_2 (\ccc P^2)$. As $i \to \infty$, the image of the $i^{\textrm{th}}$ sphere can be chosen to lie in ever-smaller balls centered at the vertex, and a similar phenomenon to that described in Example \ref{ctrex} arises.
\end{example}

\bibliographystyle{amsabbrv}
\bibliography{C:/Users/John/mybib}

\end{document}